\documentclass[12pt]{amsart}
\usepackage{fullpage}
\allowdisplaybreaks
\usepackage{amssymb}
\usepackage{eucal}
\usepackage[all]{xy}
\SelectTips{eu}{}
\theoremstyle{plain}
\newtheorem{theorem}{Theorem}[section]
\newtheorem{lemma}[theorem]{Lemma}
\newtheorem{prop}[theorem]{Proposition}
\newtheorem{cor}[theorem]{Corollary}
\theoremstyle{definition}

\newtheorem{remark}[theorem]{Remark}

\newtheorem{eg}[theorem]{Example}

\numberwithin{equation}{theorem}

\newcommand{\bA}{\mathbb A}
\newcommand{\bP}{\mathbb P}
\newcommand{\cF}{\mathcal F}
\newcommand{\cE}{\mathcal E}
\newcommand{\cO}{\mathcal O}
\newcommand{\cT}{\mathcal T}
\newcommand{\cC}{\mathcal C}
\newcommand{\cK}{\mathcal K}
\newcommand{\cI}{\mathcal I}
\newcommand{\ep}{\varepsilon}
\newcommand{\im}{\mathsf{Im}\,}
\newcommand{\Ker}{\mathsf{Ker}\,}
\newcommand{\Image}{\mathsf{Im}\,}
\newcommand{\coker}{\mathsf{Coker}\,}
\newcommand{\Rad}{\mathsf{Rad}}

\newcommand{\stmod}{\mathsf{stmod}}
\newcommand{\Z}{\mathbb Z}
\newcommand{\wt}{\widetilde}

\def\Spec{\operatorname{Spec}\nolimits}
\def\Tor{\operatorname{Tor}\nolimits}
\def\Coh{\operatorname{Coh}\nolimits}
\def\rk{\operatorname{rk}\nolimits}

\newenvironment{changemargin}[2]{%
  \begin{list}{}{%
    \setlength{\topsep}{0pt}%
    \setlength{\leftmargin}{#1}%
    \setlength{\rightmargin}{#2}%
    \setlength{\listparindent}{\parindent}%
    \setlength{\itemindent}{\parindent}%
    \setlength{\parsep}{\parskip}%
  }%
  \item[]}{\end{list}}

\author{Dave Benson}
\author{Julia Pevtsova$^{*}$}
\title[Modules of constant Jordan type and vector bundles]
{A realization theorem for modules of constant Jordan type and vector bundles}
\thanks{ $^{*}$ partially supported by the NSF}
\begin{document}
\begin{abstract}
Let $E$ be an elementary abelian $p$-group of rank $r$ and let
$k$ be a field of characteristic $p$.
We introduce functors $\cF_i$ from finitely generated 
$kE$-modules of constant Jordan type to vector bundles over projective space
$\bP^{r-1}$. 
The fibers of the functors $\cF_i$ encode complete information about
the Jordan type of the module.

We prove that given any vector bundle $\cF$
of rank $s$ on $\bP^{r-1}$,
there is a $kE$-module $M$ of stable constant Jordan type
$[1]^s$ such that $\cF_1(M)\cong \cF$
if $p=2$, and such that $\cF_1(M) \cong F^*(\cF)$ if $p$ is odd. Here,
$F\colon\bP^{r-1}\to\bP^{r-1}$ is the Frobenius map. 
We prove that the theorem cannot be improved if $p$ is odd, because
if $M$ is any module of stable constant Jordan type $[1]^s$ 
then the Chern numbers $c_1,\dots,c_{p-2}$
of $\cF_1(M)$ are divisible by $p$.
\end{abstract}
\maketitle

\section{Introduction}

The class of modules of constant Jordan type was introduced by 
Carlson, Friedlander and  the second author 
\cite{Carlson/Friedlander/Pevtsova:2008a},
and then consequently studied in 
\cite{Benson:horrocks,Benson:2010a,  
Carlson/Friedlander:2009a,Carlson/Friedlander/Suslin:rank2, 
Friedlander/Pevtsova:constr,Friedlander/Pevtsova:2010a}. 
The connection between modules of constant Jordan 
type and algebraic vector bundles on projective varieties
was first observed and developed by Friedlander and the 
second author in \cite{Friedlander/Pevtsova:constr} 
in the general setting of an arbitrary infinitesimal group scheme.  
In the present paper, we study this connection 
for an elementary abelian $p$-group. 

Let $k$ be a field of characteristic $p$ and let $E$ be an
elementary abelian $p$-group of rank $r$. We define functors
$\cF_i$ ($1\le i \le p$) 
from finitely generated $kE$-modules of constant Jordan
type to vector bundles on projective space $\bP^{r-1}$, capturing
the sum of the socles of the length $i$ Jordan blocks.  
The following is the main theorem of this paper.

\begin{theorem}\label{th:main}
Given any vector bundle $\cF$ of rank $s$ on $\bP^{r-1}$,
there exists a finitely 
generated $kE$-module $M$ of stable constant Jordan type $[1]^s$ such that
\begin{itemize}
\item[(i)] if $p=2$, then $\cF_1(M)\cong\cF$.
\item[(ii)] if $p$ is odd, then $\cF_1(M)\cong F^*(\cF)$, the
pullback of $\cF$ along the Frobenius morphism  $F\colon\bP^{r-1}\to\bP^{r-1}$.
\end{itemize}
\end{theorem}

The $kE$-modules produced this way are usually large. 
For example, in \cite{Benson:horrocks}, the first author showed how to produce
a finitely generated $kE$-module $M$ of constant Jordan type 
such that $\cF_2(M)$ is isomorphic to the rank two Horrocks--Mumford
bundle on $\bP^4$. In this case, the
construction used to prove our main theorem 
produces a module $M$ of dimension many hundred times $p^5$ plus 
two such that $\cF_1(M)\cong F^*(\cF_\mathsf{HM})$, whereas
the construction in \cite{Benson:horrocks} produces a module of 
dimension $30p^5$ of stable constant Jordan type 
$[p-1]^{30}[2]^2[1]^{26}$ such that applying $\cF_2$ gives 
$\cF_\mathsf{HM}(-2)$.

The theorem for $p=2$ may be thought of as a version of the
Bernstein--Gelfand--Gelfand correspondence
\cite{Bernstein/Gelfand/Gelfand:1978a}, since the group algebra of
an elementary abelian $2$-group in characteristic two is isomorphic to
an exterior algebra. But for $p$ odd it says something new
and interesting. In particular, it is striking that 
the $p$ odd case of Theorem \ref{th:main} cannot be strengthened to
say that $\cF_1(M)\cong \cF$. The following theorem, which is
proved in Section \ref{se:Chern}, gives limitations
on the vector bundles appearing as $\cF_1(M)$ with $M$ of stable constant
Jordan type $[1]^s$.

\begin{theorem}\label{th:cp-2}
Suppose that $M$ has stable constant Jordan type $[1]^s$.
Then $p$ divides the Chern numbers $c_m(\cF_1(M))$ for $1\le m\le p-2$.
\end{theorem}

The paper is organized as follows. In Section~\ref{se:setup} 
we give basic definitions 
of the  functors $\cF_i$ and show that applied to modules of 
constant Jordan type, they 
produce algebraic vector bundles. Section~\ref{se:prop} 
analyzes behavior of the functors $\cF_i$ 
with respect to Heller shifts and duals. 
This analysis plays a key role in the proof of  our main theorem.  
Theorems~\ref{th:main} and~\ref{th:cp-2} 
are proved in Sections~\ref{se:main} and~\ref{se:Chern} respectively.


\section{Definition of the functors $\cF_i$}\label{se:setup}

Let $k$ be a  perfect field of characteristic $p$.
Let $E=\langle g_1,\dots,g_r\rangle$ be an elementary abelian 
$p$-group of rank $r$, and set $X_i=g_i-1\in kE$ for $1\le i \le r$.
Let $J(kE) = \langle X_1, \ldots, X_r\rangle$ 
be the augmentation ideal of $kE$.
The images of $X_1,\dots,X_r$ form a basis for $J(kE)/J^2(kE)$,
which we think of as affine space $\bA_k^r$ over $k$.  
Let $K/k$ be a field extension.
If $0\ne \alpha=(\lambda_1,\dots,\lambda_r)\in\bA_K^r$, we define 
\[ X_\alpha=\lambda_1 X_1 + \dots + \lambda_r X_r \in KE. \]
This is an element of $J(KE)$ satisfying $X_\alpha^p=0$.
If $M$ is a finitely generated $kE$-module then $X_\alpha$ acts
nilpotently on $M_K = M \otimes K$, and we can decompose $M_K$ 
into Jordan blocks.
They all have eigenvalue zero, and length between $1$ and $p$.
We say that $M$ has {\em constant Jordan type\/} $[p]^{a_p}\dots [1]^{a_1}$
if there are $a_p$ Jordan blocks of length $p$, \dots, $a_1$
blocks of length $1$, independently of choice of $\alpha$.
Since $a_p$ is determined by $a_{p-1},\dots,a_1$ and the dimension
of $M$, we also say that $M$ has {\em stable constant Jordan type\/}
$[p-1]^{a_{p-1}}\dots [1]^{a_1}$. Note that the property of having 
{\it constant Jordan type} and the type itself do not depend on the choice 
of generators $\langle g_1,\dots,g_r\rangle$ 
(see \cite{Carlson/Friedlander/Pevtsova:2008a}).

We write $k[Y_1,\dots,Y_r]$ for the coordinate ring $k[\bA^r]$,
where the $Y_i$ are the linear functions defined by $Y_i(X_j)=\delta_{ij}$
(Kronecker delta). We write $\bP^{r-1}$ for the corresponding projective 
space. Let $\cO$ be the structure sheaf on $\bP^{r-1}$. If $\cF$ is
a sheaf of $\cO$-modules and $j\in\Z$, we write $\cF(j)$ 
for the $j$th Serre twist $\cF\otimes_\cO \cO(j)$.
If $M$ is a finitely generated $kE$-module, we write $\wt M$ for
the trivial vector bundle $M \otimes_k \cO$, so that $\wt M(j)=
M \otimes_k \cO(j)$. 
Friedlander and the second author \cite[\S4]{Friedlander/Pevtsova:constr} 
define a map of vector bundles $\theta_M\colon \wt M \to \wt M(1)$ 
by the formula 
\[ \theta_M(m \otimes f) = \sum_{i=1}^r X_i(m)\otimes Y_i f. \]
By abuse of notation we also write $\theta_M$ for the twist
$\theta_M(j)\colon\wt M(j)\to\wt M(j+1)$. With this convention we
have $\theta_M^p=0$.

We define functors $\cF_{i,j}$ for $0\le j <  i \le p$ 
from finitely generated $kE$-modules to coherent  sheaves on $\bP^{r-1}$ 
by taking the following subquotients of $\wt M$:
\[ \cF_{i,j}(M)=\frac{\Ker\theta_M^{j+1}\cap\im\theta_M^{i-j-1}}
{(\Ker\theta_M^{j+1}\cap\im\theta_M^{i-j})+
(\Ker\theta_M^j\cap\im\theta_M^{i-j-1})} \]
We then define
\[ \cF_i(M) = \cF_{i,0}(M)=\frac{\Ker\theta_M\cap\im\theta_M^{i-1}}
{\Ker\theta_M\cap\im\theta_M^i} \]
For a  point $0\not =\alpha \in \bA^r$ and the corresponding operator 
$X_\alpha\colon M \to M$, we also define 
\[ \cF_{i, \alpha}(M) = \frac{\Ker X_\alpha\cap\im X_\alpha^{i-1}}
{\Ker X_\alpha\cap\im X_\alpha^{i}} \]
Note that $\cF_{i,\alpha}(M)$ is evidently well-defined for 
$\bar \alpha \in \bP^{r-1}$. 

In the next Proposition we show that functors $\cF_i$ take  
modules of constant Jordan type to algebraic vector bundles 
(equivalently, locally free sheaves), 
and that they commute with specialization. 

\begin{prop}\label{p:bundles}\ 

\begin{enumerate} 
\item Let $M$ be a $kE$-module of constant 
Jordan type $[p]^{a_p}\dots [1]^{a_1}$. Then
the sheaf $\cF_i(M)$ is locally free of rank $a_i$. 
\item Let $f\colon M \to N$ be a map of modules of constant Jordan type. 
For any point $\bar\alpha=[\lambda_1:\dots:\lambda_r] \in \bP^{r-1}$ 
with residue field $k(\bar\alpha)$ we have a commutative diagram 
\[\xymatrix{ \cF_i(M) \otimes_\cO k(\bar\alpha)\ar[d]^{\simeq} 
\ar[r]^{\cF_i(f)} &\cF_i(N)\otimes_\cO k(\bar\alpha)\ar[d]^{\simeq}\\
\cF_{i,\alpha}(M) \ar[r] & \cF_{i, \alpha}(N)} \]
\end{enumerate}
\end{prop}

\begin{proof} (1). Since the  module $M$ is fixed throughout the proof, 
we shall use $\theta$ to denote $\theta_M$.

Note that $\Ker\theta\cap\im\theta^i = \Ker\{\theta\colon
\im\theta^i \to \im\theta^{i+1} \}$. 
Hence, we have a short exact sequence 
\begin{equation}\label{eq:inter}
\xymatrix{0 \ar[r]& \Ker\theta\cap\im\theta^i\ar[r]&
\im\theta^i \ar[r]^{\theta} & \im\theta^{i+1} \ar[r] & 0 }
\end{equation} 
Since $M$ has constant Jordan type, $\im\theta^{i}$ 
is locally free by \cite[4.13]{Friedlander/Pevtsova:constr}.  
Therefore, specialization of the sequence (\ref{eq:inter}) 
at any point $\bar \alpha = 
[\lambda_1: \dots :\lambda_r]$ of $\bP^{r-1}$ yields a short 
exact sequence of vector spaces 
\begin{equation}\label{eq:fibers}
0 \to (\Ker\theta\cap\im\theta^{i}) \otimes_{\cO} k(\bar \alpha) 
\to\im\theta^{i} \otimes_{\cO} k(\bar \alpha)  \to \im\theta^{i+1} 
\otimes_{\cO} k(\bar \alpha) \to 0. 
\end{equation} 
By \cite[4.13]{Friedlander/Pevtsova:constr}, 
$\im\theta^{i} \otimes_{\cO} k(\bar \alpha) 
\simeq \im \{X^i_\alpha\colon M \to M \}$. In particular, the dimension of 
fibers of $\im\theta^{i}$ 
is constant and equals $\sum\limits_{j=i+1}^p a_j(j-i)$.  
We can rewrite the sequence (\ref{eq:fibers}) as 
$$
\xymatrix{0 \ar[r]& (\Ker\theta\cap\im\theta^{i}) \otimes_{\cO} 
k(\bar \alpha) 
\ar[r]&\im X_\alpha^i  \ar[r]^{X_\alpha} & \im X_\alpha^{i+1} \ar[r] & 0 }.
$$
Hence the fiber of $\Ker\theta\cap\im\theta^i$ at a point 
$\bar \alpha$ equals $\Ker X_\alpha \cap \im X_\alpha^i$. 
In particular, $\Ker\theta\cap\im\theta^i$  
has fibers of constant dimension, equal to 
\[ \sum\limits_{j=i+1}^p a_j(j-i) - \sum\limits_{j=i+2}^p a_j(j-i-1) 
= \sum\limits_{j=i+1}^p a_j. \]
Applying \cite[4.10]{Friedlander/Pevtsova:constr} 
(see also \cite[V. ex. 5.8]{Hartshorne:1977a}), 
we conclude that $\Ker\theta\cap\im\theta^i$ is 
locally free of rank $\sum\limits_{j=i+1}^p a_j$.

Consider the short exact sequence that defines $\cF_i(M)$:
\begin{equation*}\xymatrix{0 \ar[r]& \Ker\theta\cap\im\theta^{i} \ar[r] & 
\Ker\theta\cap\im\theta^{i-1} \ar[r]&\cF_i(M)  \ar[r] & 0. }
\end{equation*}
Specializing at $\bar \alpha$, we get 
%
%
{
\[ \xymatrix@C=4mm{(\Ker\theta\cap\im\theta^i)\otimes_{\cO} 
k( \bar \alpha) \ar[r] \ar@{=}[d] & 
(\Ker\theta\cap\im\theta^{i-1})\otimes_{\cO} k( \bar \alpha)
\ar[r]\ar@{=}[d]&\cF_i(M) \otimes_{\cO} k( \bar \alpha) 
\ar[r]\ar@{=}[d] & 0 \\
\Ker X_\alpha \cap\im X_\alpha^i\ar[r] & 
\Ker X_\alpha \cap\im X_\alpha^{i-1} \ar[r]&\cF_i(M) 
\otimes_{\cO} k( \bar \alpha)  \ar[r] & 0 
} \]
}
The first arrow of the bottom row is clearly an injection. 
Hence, 
\[ \dim (\cF_i(M) \otimes_{\cO} k( \bar \alpha)) = 
\sum\limits_{j=i}^p a_j - \sum\limits_{j=i+1}^p a_j = a_i \] 
for any point $\alpha \in \bP^{r-1}$. 
Applying \cite[4.10]{Friedlander/Pevtsova:constr} again, we conclude that 
$\cF_i(M)$ is locally free (of rank $a_i$). 

Statement (2) follows immediately by applying the last diagram  
to both $M$ and $N$. 
\end{proof}

\begin{lemma}
$\wt M$ has a filtration in which the filtered quotients are
isomorphic to $\cF_{i,j}(M)$ for $0\le j<i\le p$.
\end{lemma}
\begin{proof} We consider two filtrations  on $\wt M$, the 
``kernel filtration" and the ``image filtration":
\begin{gather*}
0 \subset \Ker \theta_M \subset \ldots \subset \Ker \theta_M^{p-1} \subset 
\wt M \\
0 = \Image \theta_M^p \subset \Image \theta_M^{p-1} \subset 
\ldots \subset \Image \theta_M \subset \Image \theta_M^0=\wt M
\end{gather*}
To simplify notation, we set $\cK_j = \Ker \theta^i_M$ and 
$\cI_i = \Image \theta_M^{p-i}$.  
Using the standard refinement procedure, we refine the kernel 
filtration by the image filtration:
$$ \cK_j \subset (\cK_{j+1}\cap \cI_1) + \cK_j \subset \ldots 
\subset (\cK_{j+1}\cap \cI_{\ell}) + \cK_j \subset (\cK_{j+1}\cap \cI_{\ell+1}
) + \cK_j \subset \ldots \subset \cK_{j+1}$$  
For any three sheaves $A, B, C$  with $ B\subset A$, the second 
isomorphism theorem and the modular law imply that
\[ \frac{A+C}{B+C} \simeq \frac{A+(B+C)}{B+C} \simeq 
\frac{A}{A\cap(B+C)}\simeq \frac{A}{B+(A\cap C)}. \]
Hence, we can identify the subquotients of the refined kernel 
filtration above as
$$ \frac{(\cK_{j+1}\cap \cI_{\ell+1}) + \cK_j}{(\cK_{j+1}\cap 
\cI_{\ell}) + \cK_j} \simeq 
\frac{\cK_{j+1}\cap \cI_{\ell+1}}{(\cK_{j+1}\cap \cI_{\ell}) + 
(\cK_{j}\cap \cI_{\ell+1})}
$$
Setting $i=p-\ell+j$, we get that the latter quotient is precisely 
$\cF_{i,j}(M)$ (note that when $j > \ell$, the corresponding subquotient 
 is trivial). 
\end{proof}

\begin{lemma}\label{le:Fij}
For $0\le j<i$, we have a natural isomorphism $\cF_{i,j}(M)\cong\cF_i(M)(j)$.
\end{lemma}
\begin{proof}
For $0<j<i$, the map 
$\theta_M\colon\wt M\to\wt M(1)$ induces a natural isomorphism
$\cF_{i,j}(M)\to\cF_{i,j-1}(M)(1)$. Since $\cF_{i,0}=\cF_i$, the result follows
by induction on $j$.
\end{proof}

\begin{remark}
It follows from the proof of Proposition~\ref{p:bundles} that
the subquotient functors $\cF_{i,j}$ are linked as follows:
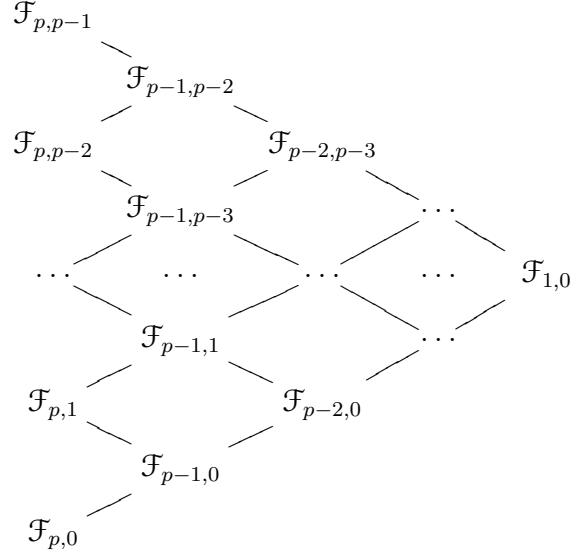
\[ \xymatrix@=2mm{\cF_{p,p-1} \ar@{-}[dr] \\ 
& \cF_{p-1,p-2} \ar@{-}[dr] \\
\cF_{p,p-2} \ar@{-}[ur] \ar@{-}[dr] & & \cF_{p-2,p-3} \ar@{-}[dr] \\
& \cF_{p-1,p-3} \ar@{-}[ur] \ar@{-}[dr] & & \ \dots\  \ar@{-}[dr] \\
\dots\ar@{-}[ur]\ar@{-}[dr]&\dots&\dots\ar@{-}[ur]\ar@{-}[dr]&\dots& 
\ \ \cF_{1,0}\ \  \\
& \cF_{p-1,1} \ar@{-}[dr] \ar@{-}[ur] & & \ \dots\  \ar@{-}[ur] \\
\cF_{p,1} \ar@{-}[dr] \ar@{-}[ur] & & \cF_{p-2,0} \ar@{-}[ur]  \\
& \cF_{p-1,0}  \ar@{-}[ur] \\
\cF_{p,0} \ar@{-}[ur]} \]
\end{remark}

We finish this section with an example.  
\begin{eg}\label{eg:tangent}
Let $M=kE/J^2(kE)$. Then $M$ has constant Jordan type $[2][1]^{r-1}$.
In the short exact sequence of vector bundles
\[ 0 \to \widetilde{M/\Rad M} \xrightarrow{\theta} 
\widetilde{\Rad M}(1) \to \cF_1(M)(1) 
\to 0 \] 
the map $\theta$ (induced by $\theta_M$) 
is equal to the map defining the tangent bundle (or sheaf of 
derivations) $\cT$ of $\bP^{r-1}$:
\[ 0 \to \cO \to \cO(1)^r \to \cT \to 0.\]
It follows that $\cF_1(M)\cong \cT(-1)$. 
On the other hand we have
$\cF_{2,1}(M) \cong \cO$, and hence $\cF_2(M) \cong \cO(-1)$.
\end{eg}


\section{Twists and syzygies}\label{se:prop}
We need a general lemma whose proof we provide for completeness.
\begin{lemma}
\label{l:crit}  Let $X$ be a Noetherian scheme over $k$, and  let
$M, N$ be  locally free $\cO_X$-modules. Let $f\colon M \to N$ be a morphism of 
$\cO_X$-modules such that  
\[ f \otimes_{\cO_X} k(x)\colon M \otimes_{\cO_X} k(x) \to 
N \otimes_{\cO_X} k(x) \] 
is an isomorphism for any $x \in X$.  Then $f$ is an isomorphism.  
\end{lemma}
\begin{proof}
It suffices to show that $f$ induces an isomorphism on stalks. 
Hence, we may assume that $X = \Spec R$, where $R$ is a local ring 
with the maximal ideal $\mathfrak m$, and $M$, $N$ are free modules.  
Since specialization is right exact, $f$ is surjective by Nakayama's 
lemma. Hence, we have an exact sequence of $R$-modules: 
\[ 0\to \ker f \to M \to N \to 0. \] 
Since $N$ is free, 
$\Tor_1^R(N, R/\mathfrak m)$ vanishes, and hence 
$\ker f \otimes_R  R/\mathfrak m =0$. By Naka\-yama's lemma, 
$\ker f=0$; therefore, $f$ is injective.
\end{proof} 
\begin{theorem}\label{th:Omega}
Let $M$ be a finite dimensional $kE$-module and let $1\le i\le p-1$.
Then there is a natural isomorphism 
$$
\cF_i(M)(-p+i) \cong \cF_{p-i}(\Omega M).
$$
\end{theorem}
\begin{proof}
Consider the diagram
\begin{equation}\label{eq:switchback} 
\xymatrix{0 \ar[r] & \widetilde{\Omega M} \ar[r] \ar[d]^{\theta_{\Omega M}} 
& \widetilde{P_M} \ar[r] \ar[d]^{\theta_{P_M}} &
\wt M \ar[r] \ar[d]^{\theta_M} & 0 \\
0 \ar[r] & \widetilde{\Omega M}(1) \ar[r] & \widetilde{P_M}(1) \ar[r] & 
\wt M(1) \ar[r] & 0, } 
\end{equation}
where $P_M$ is a projective cover of $M$. Let 
$$
\delta\colon \Ker \theta_M \to \coker \theta_{\Omega M}
$$ 
be the switchback map.  A simple diagram chase in conjunction 
with the fact that $\theta_{P_M}^p=0$ yields that 
the restriction of $\delta$ to $\Ker \theta_M\cap \im \theta^{i-1}_M$ 
lands in 
\[ \frac{\Ker \theta^{p-i}_{\Omega M}}
{\Ker \theta^{p-i}_{\Omega M} \cap \im \theta_{\Omega(M)}}(1). \] 
Projecting the latter onto 
\[ \cF_{p-i,p-i-1}(\Omega M)(1)=\frac{\Ker \theta^{p-i}_{\Omega M}}
{\Ker \theta^{p-i-1}_{\Omega M} + \Ker \theta^{p-i}_{\Omega M} 
\cap \im \theta_{\Omega(M)}}(1), \] 
we get a map of bundles:
\[ \delta\colon \Ker \theta_M\cap \im \theta^{i-1} 
\to \cF_{p-i,p-i-1}(\Omega M)(1). \]
Since $\delta$ evidently kills $\Ker \theta_M\cap \im \theta^{i}_M$, 
we conclude that $\delta$ factors through $\cF_i(M)$. Hence, 
we have an induced map
\[ \delta\colon \cF_i(M) \to \cF_{p-i,p-i-1}(\Omega M)(1). \]
A simple block count shows that this is an isomorphism at each fiber.  
Hence, by Lemma~\ref{l:crit}, this is an isomorphism of bundles.
Thus using Lemma \ref{le:Fij} (i.e., applying $\theta_{\Omega M}$ a further
$p-i-1$ times), we have
\begin{equation*}
\cF_i(M) \cong \cF_{p-i,p-i-1}(\Omega M)(1) \cong \cF_{p-i}(\Omega M)(p-i).
\end{equation*} Twisting by $\cO(-p+i)$, we get the desired isomorphism. 

Let $f\colon M \to N$ be a map of $kE$-modules. The naturality 
of the isomorphism $\cF_i(M)(-p+i)  \cong \cF_{p-i}(\Omega M)$ is 
equivalent to the commutativity of the diagram
\begin{equation}\label{eq:nat1}
\vcenter{\xymatrix
{\cF_i(M)(-p+i) \ar[d]^{\simeq}\ar[rr]^-{\cF_{i}(f)(-p+i)} 
&&\cF_i(N)(-p+i)\ar[d]^{\simeq}\\
\cF_{p-i}(\Omega M) \ar[rr]^-{\cF_{p-i}(\Omega f)}&& \cF_{p-i}(\Omega N).}}
\end{equation}
The commutativity follows from the construction of the map $\delta$ 
and naturality of the ``shifting  isomorphism"  of Lemma~\ref{le:Fij}.
\end{proof}

\begin{cor}\label{co:Omega2}
Let $M$ be a finite dimensional $kE$-module and let $1\le i\le p-1$.
Then $\cF_i(\Omega^2 M)\cong \cF_i(M)(-p)$.
\end{cor}
\begin{proof}
Apply the theorem twice.
\end{proof}

\begin{cor}\label{co:Omegank}
We have $\cF_1(\Omega^{2n}k)\cong\cO(-np)$, 
and $\cF_{p-1}(\Omega^{2n-1}k)\cong\cO(1-np)$.
\end{cor}
\begin{proof}
This follows from the theorem and the corollary,
using the isomorphism $\cF_1(k)\cong\cO$.
\end{proof}

\begin{remark}
\label{re:p2}
If $p=2$ then Theorem \ref{th:Omega} and 
Corollary \ref{co:Omegank} reduce to the statements 
that $\cF_1(\Omega M) \cong \cF_1(M)(-1)$ and $\cF_1(\Omega^nk) \cong \cO(-n)$.
\end{remark}

For a coherent sheaf $\cE$, we denote by $\cE^\vee  = \mathcal Hom_\cO(\cE, \cO)$ the dual sheaf. 
\begin{theorem}
Let $M^*$ be the $k$-linear dual of $M$, as a $kE$-module. Then
\[ \cF_i(M^*)\cong \cF_i(M)^\vee(-i+1). \]
\end{theorem}
\begin{proof}
This follows from the more obvious isomorphism 
$\cF_{i,i-1}(M^*)\cong\cF_{i,0}(M)^\vee$
together with Lemma \ref{le:Fij}.
\end{proof}

We finish this section with exactness properties of the functors $\cF_i$ 
which will be essential in the proof of the main theorem. 

Let $\cC(kE)$ be the {\it exact category of modules of constant Jordan type} 
as introduced in \cite{Carlson/Friedlander:2009a}. 
This is an exact category in the sense of Quillen: 
the objects are finite dimensional $kE$-modules of constant Jordan type, 
and the admissible morphisms are morphisms which can be completed to a 
locally split short exact sequence.  We call a sequence  of $kE$-modules 
\[ \xymatrix{0\ar[r]& M_1 \ar[r]& M_2 \ar[r]& M_3 \ar[r]& 0} \]
{\it locally split} if it is split upon restriction to 
$k[X_\alpha]/X_\alpha^p$ for any $ 0 \not = \alpha \in \bA^r$.    
\begin{prop}
\label{exact}
The functor $\cF_i\colon \cC(kE) \to \Coh(\bP_k^{r-1})$ is exact for 
$1 \leq i \leq p-1$.
\end{prop}
\begin{proof}   
Let 
\[ \xymatrix{0\ar[r]& M_1 \ar[r]& M_2 \ar[r]& M_3 \ar[r]& 0} \] 
be a locally split  short exact sequence of  modules of constant Jordan type.   
Then the Jordan type of the middle term is the sum of Jordan types of the 
end terms. Hence, $\rk \cF_i(M_2) = \rk \cF_i(M_1) + \rk \cF_i(M_3)$ for any 
$i$. 
Consider  the map $\cF_i(M_2) \to \cF_i(M_3)$. By Proposition~\ref{p:bundles}, 
the specialization $\cF_i(M_2) \otimes_\cO k(\bar \alpha) \to \cF_i(M_3) 
\otimes_\cO k(\bar \alpha)$ is surjective at any point 
$\bar \alpha \in \bP^{r-1}$. Arguing as in Lemma~\ref{l:crit}, we conclude 
that $\cF_i(M_2) \to \cF_i(M_3)$ is surjective.  Similarly, we show that 
$\cF_i(M_1) \to \cF_i(M_2)$
is injective.   Finally, the equality $\rk \cF_i(M_2) = \rk \cF_i(M_1) + 
\rk \cF_i(M_3)$ implies  exactness in the middle term. 
\end{proof}


\section{The construction}
\label{se:main}
Since $H^1(E,k)$ is the vector space dual of 
$J(kE)/J^2(kE)$, there are elements 
$y_1,\dots,y_r$ forming a vector space basis for $H^1(E,k)$
and corresponding to the linear functions $Y_1,\dots,Y_r$ on $J(kE)/J^2(kE)$
introduced in Section \ref{se:setup}.
Because of the difference in structure of the cohomology ring,
we divide the discussion 
into two cases, according as $p=2$ or $p$ is odd.\medskip

\noindent
{\bf Case 1:  $p=2$.}  
In this case the cohomology ring $H^*(E,k)$ is the polynomial
algebra $k[y_1,\dots,y_r]$. We define a $k$-algebra homomorphism
\[ \rho\colon H^*(E,k)= k[y_1,\dots,y_r] \to k[Y_1,\dots,Y_r] \]
by $\rho(y_i)=Y_i$. Recall that we have an isomorphism $\cO(-n)  = \cF_1(k)(-n) \simeq \cF_1(\Omega^nk)$ by Remark~\ref{re:p2}.

\begin{lemma}
\label{le:even}
If $\zeta\in H^n(E,k)$ is represented by a cocycle 
$\hat\zeta\colon\Omega^{n+j} k \to \Omega^j k$ (with $j\in\Z$) then the diagram
\[ \xymatrix{  
\cO(-n-j) \ar[d]^\cong \ar[r]^-{\rho(\zeta)} & \cO(-j)\ar[d]^\cong\\
\cF_1(\Omega^{n+j}k) \ar[r]^{\cF_1(\hat\zeta)} &
\cF_1(\Omega^j k) } \]
commutes.
\end{lemma}
\begin{proof} Consider $\hat \zeta\colon\Omega^{n} k \to k$. The commutative diagram~\ref{eq:nat1} applied to $\hat\zeta$ 
and iterated $j$ times becomes 
\begin{equation*}
\xymatrix
{\cF_1(\Omega^{n} k)(-j) \ar[d]^{\simeq}\ar[rr]^-{\cF_{1}(\hat \zeta)(-j)} &&\cF_1(k)(-j)\ar[d]^{\simeq}\\
\cF_1(\Omega^{n+j} k) \ar[rr]^-{\cF_1(\Omega^j \hat \zeta)}&& \cF_{1}(\Omega^j k).
}
\end{equation*} 
Hence, it suffices to assume that $j=0$. Additivity  of the 
functor $\cF_1$ allows us to assume 
that $\zeta$ is a monomial on generators $y_1, \ldots, y_r$. 
Finally, since  multiplication in 
cohomology corresponds to composition of the corresponding maps 
on Heller shifts of $k$, 
it suffices to prove our statement for a degree one generator $\zeta=y_i$.

In the case $j=0$, $\zeta=y_i$, we need to show that the following diagram 
commutes (this is the diagram above twisted by $\cO(1)$)
\begin{equation}\label{eq:j0}
\xymatrix
{\cF_1(k) \ar[d]^{\simeq}\ar[rr]^-{Y_i} &&\cF_1(k)(1)\ar@{=}[d]\\
\cF_1(\Omega k)(1) \ar[rr]^-{\cF_1(\hat y_i)}&& \cF_{1}(k)(1).
}\end{equation}

Let $E_i$ be the subgroup of index two in $E$ such that
$y_i$ is inflated from $E/E_i$ to $E$, namely the subgroup generated
by all of $g_1,\dots,g_r$ except $g_i$. Then $y_i$ represents
the class of the extension 
\[ 0 \to k \to M_i \to k \to 0 \]
where $M_i$ is the permutation module on the cosets of $E_i$.
This is a length two module on which 
$X_1,\dots,X_r$ act as zero except for $X_i$, which acts as a Jordan
block of length two.  We have a commutative diagram of $kE$-modules
\begin{equation} 
\label{eq:m}\xymatrix{0 \ar[r]& \Omega k \ar[r] \ar[d]^{y_i} & P_0 \ar[r]
\ar[d] & k \ar[r] \ar@{=}[d]& 0 \\
0 \ar[r] & k \ar[r] & M_i \ar[r] & k \ar[r] & 0} 
\end{equation}
The left vertical isomorphism $\delta\colon  \cF_1(k) \stackrel{\sim}{\to} 
\cF_1(\Omega k)(1)$  of the diagram~\ref{eq:j0}
 is given by the switchback map for the short exact sequence
\[ 0 \to \Omega k \to P_0 \to k \to 0 \]
as in diagram \eqref{eq:switchback}.   
Applying  $\theta$ to the commutative diagram on free $\cO$-modules 
induced by the module diagram~\eqref{eq:m}, we get a commutative diagram

\[ \xymatrix@=2mm{0 \ar[rr]&& \widetilde{\Omega k} \ar[rr]\ar[dr] 
\ar'[d][ddd]_{\widetilde y_i} && \widetilde{P_0} \ar[dr]\ar[rr]
\ar'[d][ddd] && \cO \ar[dr]\ar[rr] \ar@{=}'[d][ddd] && 0 &&&&\\
&0 \ar[rr]&& \widetilde{\Omega k}(1) \ar[ddd]_(0.4){\tilde y_i(1)}\ar[rr]  
&& \widetilde{P_0}(1) \ar[ddd]\ar[rr]
 && \cO(1)\ar@{=}[ddd] \ar[rr]&& 0\\ 
&&&&&&&&&\\
0\ar[rr] && \cO \ar'[r][rr]\ar[dr]&& \widetilde{M_i} 
\ar[dr]\ar'[r][rr]&& \cO \ar[dr]\ar'[r][rr]&& 0 &&&&\\
& 0 \ar[rr] && \cO(1) \ar[rr] && \widetilde{M_i}(1) \ar[rr] && \cO(1) 
\ar[rr] && 0}\]
where all horizontal arrows going back to front are given by the  operator 
$\theta$ on the corresponding 
module. The map $\wt y_i\colon \wt{\Omega k} \to \cO$ is induced by 
$y_i\colon \Omega k \to k$. To compute the 
composite $\cF_1(\hat y_i) \circ \delta$ 
we first do the switchback map of the top layer and then push 
the result down via $\tilde y_i(1)$.  
Since the diagram is commutative, we can first push down via the identity 
map of the right vertical back arrow   
and then do the switchback of the  bottom layer. Hence, the composite 
$\cF_1(\hat y_i) \circ \delta$ is given by the 
switchback map of the bottom layer; that is, of the diagram 
\[ \xymatrix{0 \ar[r] & \cO \ar[r] \ar[d]^{\theta_k} & \wt M_i \ar[r] 
\ar[d]^{\theta_{M_i}} & \cO \ar[r] \ar[d]^{\theta_k} & 0 \\
0 \ar[r] & \cO(1) \ar[r] & \wt M_i(1) \ar[r] & \cO(1) \ar[r] & 0} \] 
The left and right hand vertical maps here are zero. 
Hence, the switchback map $\delta\colon \cO \to \cO(1)$ is given by 
multiplication by $\theta_{M_i}$, which is given by
multiplication by $Y_i$ in this situation.
\end{proof}

\noindent
{\bf Case 2: $p$ is odd.} 
We write $\beta\colon H^1(E,k)\to H^2(E,k)$ for the Bockstein map, and
we set $x_i=\beta(y_i)$. In terms of Massey products, this is given by
$x_i=-(y_i,y_i,\dots,y_i)$ ($p$ terms). 

The cohomology ring is a tensor product of an
exterior algebra with a polynomial algebra:
\[ H^*(E,k)\cong \Lambda(y_1,\dots,y_r) \otimes_k k[x_1,\dots,x_r]. \]
We define a $k$-algebra homomorphism
\[ \rho\colon k[x_1,\dots,x_r]\to k[Y_1,\dots,Y_r] \]
by $\rho(x_i)=Y_i^p$.  

\begin{lemma}
\label{le:odd}
Let $\zeta$ be a degree $n$ polynomial in $k[x_1,\dots,x_r]$, regarded
as an element of $H^{2n}(E,k)$. If $\zeta$ is represented by a cocycle 
$\hat\zeta\colon\Omega^{2(n+j)} k \to \Omega^{2j} k$ then the diagram
\[ \xymatrix{ \cO(-p(n+j)) \ar[r]^(.6){\rho(\zeta)}\ar[d]^\cong 
& \cO(-pj)\ar[d]^\cong \\ 
 \cF_1(\Omega^{2(n+j)}k) \ar[r]^(.55){\cF_1(\hat\zeta)} 
 & \cF_1(\Omega^{2j} k)}\]
commutes.
\end{lemma}
\begin{proof}
The proof is similar to the proof in the case $p=2$, but more 
complicated. Again it suffices to treat the case where $\zeta=x_i$ 
and $j=0$. In other words, we need to compute the composite 
$\cF_1(\hat x_i) \circ f$ in the diagram
\[ \xymatrix
{\cF_1(k) \ar[d]^{\simeq}_{f}\ar[rr]^-{Y^p_i} &&\cF_1(k)(p)\ar@{=}[d]\\
\cF_1(\Omega^{2} k)(p) \ar[rr]^-{\cF_1(\hat x_i)}&& \cF_{1}(k)(p).
}\]
where $f\colon \cF_1(k) \to \cF_1(\Omega^2k)(p)$ 
is the isomorphism of Corollary~\ref{co:Omega2}. Let 
\[ \xymatrix{0 \ar[r] & \Omega^2 k \ar[r]   & P_1 \ar[r]  
& P_0 \ar[r] & k \ar[r] & 0}\]  
be a truncated projective resolution of $k$. 
Tracing through the proof of Theorem~\ref{th:Omega}, we see that 
$f\colon \cF_1(k) \to \cF_1(\Omega^2k)(p)$ is a composite of three maps: 
\begin{enumerate} 
\item the switchback of the top two rows of the diagram~\ref{diag1} below 
 which gives the isomorphism $\cF(k) \to \cF_{p-1,p-2}(\Omega k)(1)$,
\item followed by the isomorphism $\theta^{p-2}_{\Omega k}\colon 
\cF_{p-1,p-2}(\Omega k)(1) \stackrel{\sim}{\to} \cF_{p-1}(\Omega k)(p-1)$ 
of Lemma~\ref{le:Fij}; 
\item followed by another switchback map, now  for the bottom two rows of 
diagram~\ref{diag1}, which gives the isomorphism 
$\cF_{p-1}(\Omega k)(p-1) \simeq \cF_1(\Omega^2k)(p)$. 
 \end{enumerate}

{\small
\begin{equation}\label{diag1}\hspace{-1cm}
\xymatrix{&&0 \ar[r]& \wt {\Omega k} 
\ar[r]\ar[d]& \wt {P_0} \ar[d]\ar[r]& \cO \ar[r]\ar[d]& 0 \\ 
 &&0 \ar[r]& \wt {\Omega k}(1) \ar[d]^{\theta_{\Omega k}^{p-2}}\ar[r]& \wt 
P_0(1) \ar[r]& \cO(1) \ar[r]& 0\\
0\ar[r]&\wt {\Omega^2 k}(p-1) \ar[r]\ar[d]&\wt {P_1}(p-1)\ar[r]\ar[d]&
\wt{\Omega k}(p-1)\ar[r]\ar[d] &0&\\
0\ar[r]&\wt {\Omega^2 k}(p)\ar[r] &\wt{P_1}(p)\ar[r]&\wt{\Omega k}(p)
\ar[r] &0&}
\end{equation} 
}
Let $E_i$ be the subgroup of index $p$ such
that $x_i$ is inflated from $E/E_i$, namely the subgroup generated
by all of $g_1,\dots,g_r$ except for $g_i$. 
We let $M_i$ be the permutation module
on the cosets of $E_i$. This is a length $p$ module on which 
$X_1,\dots,X_r$ act as zero except for $X_i$, which acts as a Jordan
block of length $p$.
Then $x_i$ represents the class of the 2-fold
extension
\[ \xymatrix{0 \ar[r]& k \ar[r]& M_i \ar[r]& M_i \ar[r]& k \ar[r]& 0} \]
where the middle map is multiplication by $X_i$. 
We construct  a diagram analogous to (\ref{diag1}) for this extension:

{\small
\begin{equation}\label{diag2}\hspace{-1.4cm}
\xymatrix{&&0 \ar[r]& \wt {N_i} \ar[r]\ar[d]& \wt {M_i} 
\ar[d]\ar[r]& \cO \ar[r]\ar[d]& 0 \\ 
 &&0 \ar[r]& \wt {N_i}(1) \ar[d]^{\theta_{N_i}^{p-2}}\ar[r]& 
\wt {M_i} \ar[r]& \cO(1) \ar[r]& 0\\
0\ar[r]&\cO(p-1) \ar[r]\ar[d]&\wt {M_i}(p-1)\ar[r]\ar[d]&
\wt{N_i}(p-1)\ar[r]\ar[d] &0&\\
0\ar[r]&\cO(p)\ar[r] &\wt{M_i}(p)\ar[r]&\cO(p)\ar[r] &0.
}
\end{equation}
}
Here, $N_i = \Image\{X_i\colon M_i \to M_i\}$.  Just as in the proof of 
Lemma~\ref{le:odd}, the module diagram 
\[ \xymatrix{0 \ar[r] & \Omega^2 k \ar[r] \ar[d]^{x_i}  & 
P_1 \ar[r] \ar[d] & P_0 \ar[r] \ar[d] & k \ar[r] \ar@{=}[d] & 0 \\
0 \ar[r] & k \ar[r] & M_i \ar[r] & M_i \ar[r] & k \ar[r] & 0} \] 
induces a commutative diagram of vector bundles with (\ref{diag1})  
on top and (\ref{diag2}) at the bottom. 
Arguing as in the proof of Lemma~\ref{le:even}, 
we compute the composite $\cF_1(\hat x_i) \circ f$ by first mapping the 
rightmost $\cO$ of diagram~\ref{diag1} identically  to the rightmost $\cO$  
of diagram \ref{diag2}, and then applying our composite of a switchback, 
followed by $\theta_{N_i}^{p-2}$, followed by another switchback  
in the diagram~\ref{diag2}.  
The maps $\theta_{M_i}\colon \wt{M_i} \to \wt{M_i}(1)$ and 
$\theta_{N_i}\colon \wt{N_i} \to \wt{N_i}(1)$ are simply multiplication by 
$Y_i$.  Since the leftmost and rightmost vertical arrows in (\ref{diag2}) 
are zero, to compute the composite of the three maps involved in 
diagram~\ref{diag2}, we have to multiply first
by $Y_i$, then by $Y_i^{p-2}$, then by $Y_i$ again. Hence, $\cF_1(\hat x_i) \circ f = Y_i^p$.
\end{proof}
We are ready to prove the main theorem. 
\begin{theorem}
Given any vector bundle $\cF$ of rank $s$ on $\bP^{r-1}$,
there exists a finitely 
generated $kE$-module $M$ of stable constant Jordan type $[1]^s$ such that
\begin{itemize}
\item[(i)] if $p=2$, then $\cF_1(M)\cong\cF$.
\item[(ii)] if $p$ is odd, then $\cF_1(M)\cong F^*(\cF)$, the
pullback of $\cF$ along the Frobenius morphism  $F\colon\bP^{r-1}\to\bP^{r-1}$.
\end{itemize}
\end{theorem}
\begin{proof}
Given a vector bundle $\cF$ on $\bP^{r-1}$, using the Hilbert syzygy theorem
we can form a resolution by
sums of twists of the structure sheaf:
\[ 0 \to \bigoplus_{j=1}^{m_r} \cO(a_{r,j}) \to  \cdots \to 
\bigoplus_{j=1}^{m_1} 
\cO(a_{1,j}) \to \bigoplus_{j=1}^{m_0} \cO(a_{0,j}) \to \cF \to 0. \]
Each of the maps in this
resolution is a matrix whose entries are homogeneous polynomials in 
$Y_1,\dots,Y_r$. Replacing each $Y_i$ with $x_i$ gives matrices of 
cohomology elements which we may use to form a sequence of modules and homomorphisms, which 
takes the form
\[ 0 \to \bigoplus \Omega^{-\ep a_{r,j}}(k) \to \dots \to
\bigoplus \Omega^{- \ep a_{1,j}}(k) \to \bigoplus \Omega^{- \ep a_{0,j}}(k)  \]
where $\ep=1$ if $p=2$ and $\ep=2$ if $p$ is odd.
This sequence is a complex in the stable module category
$\stmod(kE)$. We complete the first map to a triangle whose third object
we call $M_{r-1}$:
\[ \bigoplus \Omega^{-\ep a_{r,j}}(k) \to \bigoplus \Omega^{-\ep a_{r-1,j}}(k)
\to M_{r-1}.  \]
Since the first two entries in the triangle are modules of 
trivial stable constant Jordan type, 
the same must be true for $M_{r-1}$.  Moreover, the short exact 
sequence corresponding  to this triangle 
must be locally split. Continuing by downwards induction on $i$ 
from $i=r-2$ to $i=0$ 
we complete triangles
\[ M_{i+1} \to \bigoplus \Omega^{-\ep a_{i,j}}(k) \to M_i \]
and then finally we set $M=M_0$. By construction, $M_0$ is a module of 
trivial stable constant Jordan type, 
and all intermediate triangles correspond to locally  split sequences 
of modules of constant Jordan type. Applying $\cF_1$ to this construction,  
we obtain an exact 
sequence of vector bundles by Proposition~\ref{exact}. 
If $p=2$ it is isomorphic to the original resolution by Lemma~\ref{le:even} 
and so we have $\cF_1(M)\cong \cF$. 
On the other hand, if $p$ is odd, each of the original matrices has been
altered by replacing the variables $Y_i$ by their $p$th powers by 
Lemma~\ref{le:odd}. This is
the pullback of the original resolution along the Frobenius map
$F \colon \bP^{r-1} \to \bP^{r-1}$, and so it is a resolution of $F^*(\cF)$.
So in this case we have $\cF_1(M)\cong F^*(\cF)$.
\end{proof}

\begin{remark}
The construction given above is not functorial, despite appearances.
The problem is that given a commutative square in $\stmod(kG)$ in which
the vertical maps are isomorphisms, it may be completed to an isomorphism
of triangles, but the third arrow is not unique.
\end{remark}

\section{Chern numbers}\label{se:Chern}

Recall that the Chow ring of $\bP^{r-1}$ is
\[ A^*(\bP^{r-1})\cong \Z[h]/(h^r). \]
If $\cF$ is a vector bundle on $\bP^{r-1}$, we write
\[ c(\cF,h)=\sum_{j\ge 0}c_j(\cF)h^j\in A^*(\bP^{r-1}) \]
for the Chern polynomial, where $c_0(\cF)=1$ and the $c_i(\cF)\in\Z$
($1\le i\le r-1$) are the Chern numbers of $\cF$.

If $0\to\cF\to\cF'\to\cF''\to 0$ is a short exact sequence
of vector bundles then we have the Whitney sum formula
\[ c(\cF',h)=c(\cF,h)c(\cF'',h). \]

\begin{lemma}\label{le:chern-twist}
The formula for Chern numbers of twists of a rank $s$
vector bundle is
\[ c_m(\cF(i))=\sum_{j=0}^m i^j\binom{s-m+j}{j}c_{m-j}(\cF). \]
Equivalently, the total Chern class of the twists is given by
\begin{equation}\label{eq:cFih} 
c(\cF(i),h)= \sum_{n= 0}^s c_n(\cF)h^n(1+ih)^{s-n} 
\end{equation}
\end{lemma}
\begin{proof}
See Fulton \cite[Example 3.2.2]{Fulton:1984a}.
\end{proof}

More explicitly,
\begin{align*}
c_1(\cF(i)) &= c_1(\cF) + is \\
c_2(\cF(i)) &= c_2(\cF) + i(s-1)c_1(\cF) + i^2\binom{s}{2} \\
c_3(\cF(i)) &= c_3(\cF) + i(s-2)c_2(\cF) + i^2\binom{s-1}{2}c_1(\cF) 
+ i^3\binom{s}{3}
\end{align*}
and so on. 

\begin{lemma}\label{le:product-twists}
For a vector bundle $\cF$ of rank $s$ on $\bP^{r-1}$ we have
\[ c(\cF,h)c(\cF(1),h)\cdots c(\cF(p-1),h) \equiv 1-sh^{p-1} \pmod{(p,h^p)}. \]
\end{lemma}
\begin{proof}
We write 
\[ c(\cF)=\prod_{j=1}^s(1+\alpha_j h), \] 
where the $\alpha_j$ are the Chern roots. 
Then the formula \eqref{eq:cFih} is equivalent to
\[ c(\cF(i)) = \prod_{j=1}^s(1+(\alpha_j+i)h). \]
Thus we have
\[ c(\cF)c(\cF(1))\cdots c(\cF(p-1)) 
=\prod_{j=1}^s (1+\alpha_j h)(1+(\alpha_j+1)h)\cdots(1+(\alpha_j+p-1)h). \]
Now by Fermat's little theorem, we have the identity
\[ x(x+y)\cdots(x+(p-1)y)\equiv x^p-xy^{p-1} \pmod{p} \]
and so putting $x=1+\alpha_j h$, $y=h$ we obtain
\begin{align*} 
c(\cF)c(\cF(1))\cdots c(\cF(p-1)) 
&\equiv \prod_{j=1}^s ((1+\alpha_jh)^p - (1+\alpha_jh))h^{p-1} \pmod{p} \\
&\equiv \prod_{j=1}^s (1-h^{p-1}+(\alpha_j^p-\alpha_j)h^p) \pmod{p} \\
&\equiv 1-sh^{p-1} \pmod{(p,h^p)}.
\end{align*}
\emph{A priori}, this is a congruence between
polynomials with algebraic integer coefficients.
But if two rational integers are congruent mod $p$ as algebraic integers
then they are also congruent modulo $p$ as rational integers. This
is because their difference, divided by $p$, is both an algebraic integer
and a rational number, therefore an integer.
\end{proof}

We now restate and prove Theorem~\ref{th:cp-2}.
\begin{theorem} 
Suppose that $M$ has stable constant Jordan type $[1]^s$.
Then $p$ divides the Chern numbers $c_m(\cF_1(M))$ for $1\le m\le p-2$.
\end{theorem}
\begin{proof} 
Since $M$ has stable Jordan type $[1]^s$, we have $\cF_2(M) = \ldots = \cF_{p-1}(M) =0$. Hence,
the trivial vector bundle
$\wt M$ has a filtration with filtered quotients (not in order)
$\cF_1(M)$, $\cF_p(M)$, $\cF_p(M)(1)$, \dots, $\cF_p(M)(p-1)$.
So we have
\begin{align*} 
1 &= c(\wt M,h) \\
&= c(\cF_1(M),h)c(\cF_p(M),h)c(\cF_p(M)(1),h)\cdots c(\cF_p(M)(p-1),h) \\
&\equiv c(\cF_1(M),h) \pmod{(p,h^{p-1})}
\end{align*}
by Lemma \ref{le:product-twists}. It follows that the coefficients 
$c_m(\cF_1(M))$ are divisible by $p$ for $1\le m\le p-2$.
\end{proof}

\begin{remark} For $p=2$ this theorem says nothing. But for $p$ odd, 
it at least forces
$c_1(\cF_1(M))$ to be divisible by $p$. As an explicit example, the 
twists of the Horrocks--Mumford bundle  $\cF_\mathsf{HM}(i)$ have
$c_1=2i+5$ and $c_2=i^2+5i+10$ (\cite{Horrocks/Mumford:1973a}). 
For $p\ge 7$ these cannot both
be divisible by $p$, and so there is no module $M$ of stable constant 
Jordan type $[1]^2$ and integer $i$ such that $\cF_1(M)\cong\cF_\mathsf{HM}(i)$.
\end{remark}

\begin{remark}  The conclusion of the theorem is limited to the modules 
of stable constant Jordan type $[1]^s$. For example, if $M_n$ is a 
``zig-zag" module of dimension $2n+1$ for $\Z/p \times \Z/p$, then 
$\cF_1(M_n) \simeq \cO(-n)$ and $\cF_1(M_n^*) \simeq \cO(n)$  for any 
$n\geq 0$ (see \cite[\S 6]{Friedlander/Pevtsova:constr}).
\end{remark}

\noindent
{\bf Acknowledgement.} Both authors are grateful to MSRI for its 
hospitality during the 2008 program ``Representation theory of 
finite groups and related topics'' where much of this research was
carried out.

\bibliographystyle{amsplain}
\bibliography{../repcoh}

\providecommand{\bysame}{\leavevmode\hbox to3em{\hrulefill}\thinspace}
\providecommand{\MR}{\relax\ifhmode\unskip\space\fi MR }
\providecommand{\MRhref}[2]{%
  \href{http://www.ams.org/mathscinet-getitem?mr=#1}{#2}
}
\providecommand{\href}[2]{#2}
\begin{thebibliography}{10}

\bibitem{Benson:horrocks}
D.~J. Benson, \emph{{Modules of constant Jordan type and the Horrocks--Mumford
  bundle}}, preprint, 2008.

\bibitem{Benson:2010a}
\bysame, \emph{{Modules of constant Jordan type with one non-projective
  block}}, Algebras and Representation Theory \textbf{13} (2010), 315--318.

\bibitem{Bernstein/Gelfand/Gelfand:1978a}
I.~N. Bern{\v{s}}te{\u\i}n, I.~M. Gel$'$fand, and S.~I. Gel$'$fand,
  \emph{{Algebraic vector bundles on {$\mathbb P^n$} and problems of linear
  algebra}}, Funktsional. Anal. i Prilozhen. \textbf{12} (1978), no.~3, 66--67.

\bibitem{Carlson/Friedlander:2009a}
J.~F. Carlson and E.~M. Friedlander, \emph{{Exact category of modules of
  constant Jordan type}}, Algebra, arithmetic and geometry: Manin Festschrift,
  Progr. in Math., vol. 269, Birkh\"auser Verlag, Basel, 2009, pp.~259--281.

\bibitem{Carlson/Friedlander/Pevtsova:2008a}
J.~F. Carlson, E.~M. Friedlander, and J.~Pevtsova, \emph{{Modules of constant
  Jordan type}}, J.~Reine \& Angew.\ Math. \textbf{614} (2008), 191--234.

\bibitem{Carlson/Friedlander/Suslin:rank2}
J.~F. Carlson, E.~M. Friedlander, and A.~A. Suslin, \emph{{Modules for $\mathbb
  Z/p \times \mathbb Z/p$}}, to appear.

\bibitem{Friedlander/Pevtsova:constr}
E.~M. Friedlander and J.~Pevtsova, \emph{{Constructions for infinitesimal group
  schemes}}, to appear.

\bibitem{Friedlander/Pevtsova:2010a}
\bysame, \emph{{Generalized support varieties for finite group schemes}},
  Documenta Math. \textbf{Extra Volume Suslin} (2010), 197--222.

\bibitem{Fulton:1984a}
W.~Fulton, \emph{{Intersection theory}}, Ergebnisse der Mathematik und ihrer
  Grenzgebiete, Folge 3, Band 2, Springer-Verlag, Ber\-lin/New York, 1984.

\bibitem{Hartshorne:1977a}
R.~Hartshorne, \emph{{Algebraic geometry}}, Graduate Texts in Mathematics,
  vol.~52, Springer-Verlag, Ber\-lin/New York, 1977.

\bibitem{Horrocks/Mumford:1973a}
G.~Horrocks and D.~Mumford, \emph{{A rank $2$ vector bundle on $\mathbb P^4$
  with $15,000$ symmetries}}, Topology \textbf{12} (1973), 63--81.

\end{thebibliography}

\end{document}